\documentclass[reqno]{llncs}
\usepackage[utf8]{inputenc}
\usepackage{amsmath}
\usepackage{amsfonts}
\usepackage{amssymb}
\usepackage[svgnames]{xcolor}
\usepackage{colortbl}

\normalsize
\makeatletter
\let\bfseries=\undefined
\DeclareRobustCommand\bfseries
{\not@math@alphabet\bfseries\mathbf
  \boldmath\fontseries\bfdefault\selectfont}
\makeatother

\newcommand{\Q}{{\mathbb Q}}
\newcommand{\Z}{{\mathbb Z}}

\newcommand{\R}{\mathbb R}

\DeclareMathOperator{\diam}{diam}
\DeclareMathOperator{\dist}{dist}

\def\Circuits{{\mathcal C}}

\usepackage{enumerate}
\def\ve#1{\mathchoice{\mbox{\boldmath$\displaystyle\bf#1$}}
{\mbox{\boldmath$\textstyle\bf#1$}}
{\mbox{\boldmath$\scriptstyle\bf#1$}}
{\mbox{\boldmath$\scriptscriptstyle\bf#1$}}}

\newcommand\veb{{\ve b}}
\newcommand\vecc{{\ve c}}

\newcommand\veg{{\ve g}}

\newcommand\veu{{\ve u}}
\newcommand\vev{{\ve v}}
\newcommand\vew{{\ve w}}
\newcommand\vex{{\ve x}}
\newcommand\vey{{\ve y}}
\newcommand\vez{{\ve z}}

\newcommand{\eoproof}{\hspace*{\fill} $\square$ \vspace{5pt}}

\newcommand{\cupdot}{\mathbin{\dot{\cup}}}

\usepackage{ifthen}
\usepackage{tikz}
\usepackage{tkz-graph}
\usepackage{subfigure}
\usepackage{caption}
\usepackage{here}
\usepackage{paralist}

\newcommand{\DeclareBracket}[3]{
  \newcommand{#1}[2][]{%
  \ifthenelse%
  {\equal{##1}{}}%
  {\left#2##2\right#3}%
  {\csname ##1l\endcsname#2##2\csname ##1r\endcsname#3}}}
\DeclareBracket\set\{\}

\usepackage[colorlinks=true,breaklinks=true,bookmarks=true,urlcolor=blue,
     citecolor=blue,linkcolor=blue,bookmarksopen=false,draft=false]{hyperref}

\begin{document}

\title{Quadratic diameter bounds for dual network flow polyhedra}

\author{Steffen Borgwardt \and Elisabeth Finhold \and Raymond Hemmecke}

\institute{S. Borgwardt, Technische Universit\"at M\"unchen, +49-89-28916876, \email{\href{mailto:borgwardt@ma.tum.de}{borgwardt@ma.tum.de}}
\and E. Finhold, Technische Universit\"at M\"unchen, +49-89-28916891, \email{\href{mailto:finhold@tum.de}{finhold@tum.de}}
\and R. Hemmecke, Technische Universit\"at M\"unchen, +49-89-28916864, \email{\href{mailto:hemmecke@tum.de}{hemmecke@tum.de}}}

\date{\today}

\maketitle

\begin{abstract}
Both the combinatorial and the circuit diameters of polyhedra are of interest to the theory of linear programming for their intimate connection to a best-case performance of linear programming algorithms. 

We study the diameters of dual network flow polyhedra associated to $b$-flows on directed graphs $G=(V,E)$ and prove quadratic upper bounds for both of them: the minimum of $(|V|-1)\cdot |E|$ and $\frac{1}{6}|V|^3$ for the combinatorial diameter, and $\frac{|V|\cdot (|V|-1)}{2}$ for the circuit diameter. The latter strengthens the cubic bound implied by a result in [De Loera, Hemmecke, Lee; 2014].

Previously, bounds on these diameters have only been known for bipartite graphs. The situation is much more involved for general graphs. In particular, we construct a family of dual network flow polyhedra with members that violate the circuit diameter bound for bipartite graphs by an arbitrary additive constant. Further, it provides examples of circuit diameter $\frac{4}{3}|V| - 4$.
\end{abstract}



\section{Introduction}
In the context of a best-case performance of the Simplex algorithm, the studies of the combinatorial diameter of polyhedra are a classical field in the theory of linear programming. In particular, if one can find an $n$-dimensional polyhedron with $f$ facets with a diameter that is exponential in $f$ and $n$, then the existence of a polynomial pivot rule for the Simplex algorithm would be disproved.

In 1957, Hirsch stated the famous conjecture \cite{d-63} claiming that the combinatorial diameter of a polyhedron is at most  $f-n$.  For (unbounded) polyhedra there are low-dimensional counterexamples \cite{klee67}. For polytopes however, the Hirsch conjecture stood for more than 50 years, until Santos gave a first counterexample \cite{santos11}.
 Nonetheless the bound holds for several well-known families of polyhedra, like $0\slash 1$-polytopes \cite{n-89} or dual transportation polyhedra \cite{balinski84}. However, it is still unsolved for many classes of polyhedra, e.g. primal transportation polytopes; see \cite{kd-13}. Even the polynomial Hirsch conjecture that asks whether there is an upper bound on the combinatorial diameter of general polytopes that is polynomial in $f$ and $n$ is open. See the survey by Kim and Santos for the current state-of-the-art \cite{ks-10}. 

For our discussion, we use the following notation. Let $\vev^{(1)}$ and $\vev^{(2)}$ be two vertices of a polyhedron $P$. We call a sequence of vertices $\vev^{(1)}=\vey^{(0)},\ldots,\vey^{(k)}=\vev^{(2)}$ an \emph{edge walk of length $k$} if every pair of consecutive vertices is connected by an edge. The \emph{(combinatorial) distance} of $\vev^{(1)}$ and $\vev^{(2)}$ is the minimum length of an edge walk between $\vev^{(1)}$ and $\vev^{(2)}$. The \emph{combinatorial diameter} $\diam(P)$ of $P$ then is the maximum distance between any two vertices of $P$.

On such edge walks we only go along edges of the polyhedron $P$, in particular we never leave its boundary. In contrast to this, \emph{circuit walks} also use only 'potential' edge directions, but may walk through the interior of the polyhedron: Let a polyhedron $P$ be given by 
	\[	
		P= \set{\,
			\vez\in \R^n:\ A^1\vez=\veb^1, A^2\vez\geq \veb^2
		\,}
	\]
for matrices $A^i\in\Q^{d_i\times n}$ and vectors $\veb^i\in \R^{d_i}$, $i=1,2$.
The \emph{circuits} or \emph{elementary vectors} $\Circuits(A^1,A^2)$ of $A^1$ and $A^2$ are those vectors $\veg \in \ker(A^1)\setminus\set{\,\ve 0\,}$, for which $A^2\veg$ is support-minimal in the set $\left\{A^2\vex: \ \vex \in \ker\left(A^1\right)\backslash \{0\}\right\}$, where $\veg$ is normalized to coprime integer components. It can be shown that the set of circuits consists exactly of all edge directions of $P$ for varying $\veb^1$ and $\veb^2$ \cite{Graver:75}. Circuits and their integer programming equivalents, Graver bases, play an important role in the theory of integer programming. We refer the reader to the book \cite{DHKbook} for a thorough introduction to the topic.

The circuit analogues to the notions of combinatorial distance and diameter for $P$ are then defined as follows \cite{CircuitDiam}: For two vertices $\vev^{(1)},\vev^{(2)}$ of $P$, we call a sequence $\vev^{(1)}=\vey^{(0)},\ldots,\vey^{(k)}=\vev^{(2)}$ a \emph{circuit walk of length $k$} if for all $i=0,\ldots,k-1$ we have
\begin{enumerate}
\item $\vey^{(i)}\in P$,
\item $\vey^{(i+1)}-\vey^{(i)}=\alpha_i\veg^{(i)}$ for some $\veg^{(i)}\in\Circuits(A^1,A^2)$  and $\alpha_i>0$, and
\item $\vey^{(i)}+\alpha\veg^{(i)}$ is infeasible for all $\alpha>\alpha_i$. 
\end{enumerate}  
 The \emph{circuit distance} $\dist_\Circuits(\vev^{(1)},\vev^{(2)})$ from $\vev^{(1)}$ to $\vev^{(2)}$ then is the minimum length of a circuit walk from $\vev^{(1)}$ to $\vev^{(2)}$. The \emph{circuit diameter} $\diam_\Circuits(P)$ of $P$ is the maximum circuit distance between any two vertices of $P$.

Clearly, the circuit diameter of a polyhedron is at most as large as the combinatorial diameter of the polyhedron, as a walk along the $1$-skeleton/edges of the polyhedron is a special circuit walk. Once again, if there exists a polyhedron with exponential circuit diameter, there can be no polynomial pivot rule for the Simplex algorithm.  This is one of several reasons to study it in the context of linear programming; see \cite{CircuitDiam}. In fact, the circuit diameter gives a lower bound for any augmentation algorithm along circuit directions \cite{DHKbook}.

In fact, it is open whether there is a polyhedron with a circuit diameter that exceeds $f-n$, as in the Hirsch conjecture (see Conjecture $1$ in \cite{CircuitDiam}). The polyhedra giving counterexamples to the Hirsch conjecture do not violate this bound for the circuit diameter. This raises the natural question how these two diameters are related to one another. In this paper, we study the diameters for the family of \emph{dual network flow polyhedra}, for which we prove quadratic upper bounds on both the combinatorial diameter and the circuit diameter.

Let $G=(V,E)$ be a directed connected graph and let $A\in \set{\, -1,0,1\, }^{|V|\times|E|}$ be its node-arc incidence matrix, where $a_{ie}=-1$ and $a_{je}=1$ if arc $e$ has node $i$ as its tail and node $j$ as its head. Let $\veb\in \R^{|V|}$. A $\veb$-flow on $G$ is given by any solution to $A\vex=\veb,\ \vex\geq \ve0$, that is, in each node $i\in V$ the resulting flow (incoming minus outgoing flow) is given by $b_i$. For some cost function $\vecc\colon E\to \R_+$, the min-cost $\veb$-flow problem and its dual are  given by 
\[
	\min \set{\, \vecc^T\vex: \ A\vex=\veb,\ \vex\geq \ve0\,} \quad \text{and} \quad \max \set{\, \veu^T\veb: \ A^T\veu\leq \vecc\,}.
\]
In the following we are interested in the \emph{dual network flow polyhedron} associated to some graph $G$ and vector $\vecc\in\R^{|E|}$. These polyhedra can be written as
\[
  P_{G,\vecc}=\set{\,\veu\in\R^{|V|}:-u_a+u_b\leq c_{ab}\ \forall\ ab\in E, u_0=0\,}.
\]
As is standard, we set $u_0=0$ to make $P_{G,\vecc}$ pointed (to actually have vertices). Then linear programming over $P_{G,\vecc}$ is a viable approach for solving the corresponding min-cost $\veb$-flow problem and is another reason for the interest in the diameters of this family of polyhedra.

In \cite{balinski84} and \cite{CircuitDiam} the diameters of dual transportation polyhedra were studied. They are associated to undirected bipartite graphs and can be interpreted as dual network flow polyhedra on directed bipartite graphs on node sets $V=V_1\cupdot V_2$, where all edges point from $V_1$ to $V_2$. Hence these diameter results transfer to special cases of our more general setting:

Balinski \cite{balinski84} proved that the combinatorial diameter of a dual transportation polyhedron associated with a complete bipartite graph on $M\times N$ nodes is bounded above by $(M-1)(N-1)$ and that this bound is sharp for all $M,N$. Observe that this bound is quadratic in the number of nodes and linear in the number of edges. The circuit diameter of a dual transportation polyhedron defined on an arbitrary bipartite graph on $M\times N$ nodes is bounded above by $M+N-2=|V|-2$ (\cite{CircuitDiam}) and there are examples having circuit diameter $M+N-3=|V|-3$ for any value of $M+N$.

For general graphs, we cannot expect similar bounds. The following example gives a graph for which the upper bound $|V|-2$ does not hold for the circuit diameter.

\begin{example}\label{THEex}The dual network flow polyhedron $P_{G,\vecc}$ associated with the following graph on $4$ nodes has circuit diameter at least $|V|=4$. (See Section \ref{sec:ex} for a proof.) The edges are labeled with the corresponding values of $\vecc$.
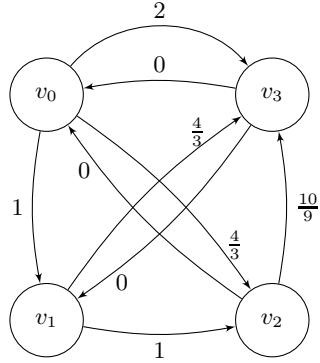
\begin{figure}[H]
\centering  
\begin{tikzpicture}
	\tikzset{vertex/.style = {shape=circle,draw,minimum size=3em}}
  \node[vertex] (v0) at  (0,3) {$v_0$};
  \node[vertex] (v1) at  (0,0) {$v_1$};
  \node[vertex] (v2) at  (3,0) {$v_2$};
  \node[vertex] (v3) at  (3,3) {$v_3$};
  
  \tikzset{edge/.style = {->,> = latex'}}
  \draw[edge, ->] (v3) [bend right = 10] to  node[above] {$0$} (v0);
  \draw[edge, ->] (v2) [bend left = 10] to (v0);
  \node at (0.5,2) {$0$};
  \draw[edge, ->] (v3) [bend left = 10] to (v1);
	\node at (1,0.5) {$0$};
  \draw[edge, ->] (v0) [bend left = 50] to [above] node {$2$} (v3);
  \draw[edge, ->] (v0) [bend left=10] to (v2);
  \node at (2.5,1) {$\frac{4}{3}$};
  \draw[edge, ->] (v1) [bend left=10] to (v3);
  \node at (2,2.5) {$\frac{4}{3}$};
  \draw[edge, ->] (v0) [bend right=10] to [left] node {$1$} (v1);
  \draw[edge, ->] (v1) [bend right=10] to [below] node {$1$} (v2);
  \draw[edge, ->] (v2)
   [bend right=10] to [right] node {$\frac{10}{9}$} (v3);
\end{tikzpicture}
\end{figure}
\eoproof
\end{example}

We extend this graph to a family of graphs with associated polyhedra of circuit diameter greater than $|V|+k-1$ for any $k$. To do so, we introduce what we call a \emph{glueing construction}: If we glue $k$ graphs together at a single, arbitrary node, we obtain a larger graph. The circuit diameter, respectively combinatorial diameter, of this larger graph then is the sum of the circuit diameters, respectively combinatorial diameters, of the polyhedra associated to the smaller graphs; see Lemma \ref{lemmaGlueing} in Section \ref{sec:ex}.

Applying this construction to $k$ copies of Example \ref{THEex} above, we get a family of graphs on $3k+1$ nodes with associated dual network flow polyhedra that admit a circuit diameter of at least $4k=|V|+k-1$. Hence we violate the circuit diameter bound for bipartite graphs by an arbitrary additive constant. This further yields a family of polyhedra whose circuit diameter approaches  $\frac{4}{3} |V|$:
\begin{lemma}\label{lemma: lower bound circuit diameter}
For any $n\geq 4$, there is a graph $G=(V,E)$ on $|V|=n$ nodes and a vector $\vecc \in \R^{|E|}$ such that 
	\[ \diam_\Circuits\left(P_{G,\vecc}\right) \geq \frac{4}{3} |V|-4\]
\end{lemma}

Thus our more general framework of arbitrary graphs is much more involved than the one for bipartite graphs. The key results of this paper are the following two theorems that, roughly speaking, tell us that turning to general graphs adds a factor of $|V|$ on the previous diameter bounds. Hence we get quadratic upper bounds on both the combinatorial and the circuit diameter.
\begin{theorem}[Combinatorial diameter]\label{thm:CombDiam}
The combinatorial diameter of dual network flow polyhedra $P_{G,\vecc}$ is bounded above by $min\{\left(|V|-1\right)\cdot |E|,\frac{|V|^3}{6}\}$.
\end{theorem}

\begin{theorem}[Circuit diameter]\label{thm:CirDiam}
The circuit diameter of dual network flow polyhedra $P_{G,\vecc}$ is bounded above by $\frac{|V|\cdot \left(|V|-1\right)}{2}$.
\end{theorem}

Theorem \ref{thm:CirDiam} strengthens the cubic bound implied by Corollary 5 in \cite{DeLoera+Hemmecke+Lee:14}.

The vertices, edges, and circuits of a dual network flow polyhedron reveal a lot of combinatorial structure. In Section \ref{sec:tools}, we provide some basic results on their graph-theoretical interpretation and use it to prepare some tools for the proofs of our main theorems. The proofs themselves then are found in Section \ref{sec:proofs}. In Section \ref{sec:ex}, we conclude the paper with a formal introduction of our glueing construction and by turning to a more detailed analysis of Example \ref{THEex} and the resulting family of polyhedra. 

\bigskip

\section{Basic results and tools}\label{sec:tools}

Throughout this paper, we will exploit the special structure of dual network flow polyhedra $P_{G,\vecc}$ by relating the vertices and edges of such polyhedra to subgraphs of the defining graph $G$. For $\veu\in P_{G,\vecc}$, we denote by $G(\veu)$ the graph with nodes $V$ and with edges $ab\in E$ for which $-u_a+u_b\leq c_{ab}$ is tight. If the polyhedron $P_{G,\vecc}$ is non-degenerate, these graphs have no cycles. 

The \emph{vertices} of $P_{G,\vecc}$ are determined by the sets of inequalities $-u_a+u_b\leq c_{ab}$ that are tight. It can be shown that that $\veu\in P_{G,\vecc}$ is a vertex if and only if $G(\veu)$ is a \emph{spanning} subgraph of $G$.
In particular, every such spanning subgraph contains a spanning tree of $G$ with $|V|-1$ edges corresponding to (a subset of) the inequalities $-u_a+u_b\leq c_{ab}$ that are tight at the vertex. This spanning tree uniquely determines the vertex $\veu$, since we assume $u_0=0$.

The \emph{circuit directions} of $P_{G,\vecc}$ can be described as follows: Let $R,S\subseteq V$ be connected nonempty node sets with $R \cupdot S=V$  (which implies $R\cap S=\emptyset$). W.l.o.g., we may assume $0\in R$. Then the vector $\veg\in\R^{M+N}$ with
\begin{equation}\label{Eq: Construction of circuit direction from R and S.}
  g_i=\left\{\begin{array}{ll}0, & \text{if } i\in R,\\ 1, & \text{if } i\in S,\end{array}\right.
\end{equation}
is an edge direction of $P_{G,\vecc}$ for some right-hand side $\vecc$. In fact, it can be shown that these are all potential edge directions and hence they constitute the set of circuits $\Circuits_G$ associated to the matrix defining $P_{G,\vecc}$.

Let $\vey \in P_{G,\vecc}$. We apply a circuit step given by $R \cupdot S=V$ or the corresponding $\veg$ by setting $\vey':= \vey \pm \epsilon \veg$, where $\epsilon$ is the smallest non-negative number such that an inequality $-u_a+u_b\leq c_{ab}$ with $a\in R$ and $b\in S$ (respectively $b\in R$ and $a\in S$) becomes tight. This means that we increase  (respectively decrease) all components $y_s$ with $s\in S$ until an edge from $R$ to $S$ (respectively from $S$ to $R$) is inserted.

Two vertices $\veu^{(1)}, \veu^{(2)}$ of $P_{G,\vecc}$ are connected by an edge if and only if the subgraph of $G$ with edge set $E\left(G(\veu^{(1)})\right)\cap E\left(G(\veu^{(2)})\right)$ consists of exactly two connected components. 
Then the node sets $R$ and $S$ of these components describe the edge direction via Equation (\ref{Eq: Construction of circuit direction from R and S.}). 
 
We continue with some advanced tools and results that we will need in Section \ref{sec:proofs}. 
The idea of \emph{contracting edges} simplifies the proofs of Theorems \ref{thm:CombDiam} and \ref{thm:CirDiam}:
Assume that we have a vertex (feasible point) $\vey$ of a polyhedron $P_{G,\vecc}$ from which we want to construct an edge walk (circuit walk) to some vertex $\vew$, and assume that $E(G(\vey))$ and $E(G(\vew))$ have an edge $ab$ in common. Then we wish to keep this edge on the remaining edge walk (circuit walk). Therefore, the difference between $u_a$ and $u_b$ has to remain constant, which means that in every edge step (circuit step) given by $V=R\cupdot S$, $a$ and $b$ are assigned both to $R$ or both to $S$. To simplify this idea, we interpret $a$ and $b$ as one node in the following sense:  We \emph{contract} the edge $ab$  and continue our edge walk (circuit walk) in a smaller polyhedron defined on a graph with one node less and adjusted edge set.

Geometrically this corresponds to intersecting the dual network flow polyhedron with the hyperplane $\left\{\veu \in \R^{|V|}\colon -u_a+u_b = c_{ab}\right\}$. This defines a face of the polyhedron, which is a dual network flow polyhedron in its own right. We then continue the edge walk (circuit walk) on this face.
More formally, let $ab$ be the common edge in $G=(V,E)$. The new polyhedron $P_{G',c'}$ is defined by a new graph $G'=(V',E')$ and a new vector $c'$ (for a simple notation we use $c_{ij}=\infty$ if $ij\notin E$) as follows:
\begin{align*}
	V'=&V\backslash \{b\} \\
	E'=&\left\{ij: ij\in E \text{ and } i,j\neq a,b \right\} \\
		 &\cup \left\{aj: aj \in E \text{ or } bj \in E\right\} 
			\cup \left\{ia: ia \in E \text{ or } ib \in E\right\}\\
c'_{ij}=&\begin{cases}
			c_{ij} & \text{for }\;  i,j\neq a, \;ij\in E' \\
			\min\left\{c_{aj}, c_{bj}+c_{ab}\right\} & \text{for }\; i=a,\; aj\in E'\\ 
			\min\left\{c_{ia} + c_{ab}, c_{ib}\right\} & \text{for }\; j=a,\; ia\in E'
	\end{cases}
\end{align*}
For the definition of $\vecc'$, observe that if $ab$ exists in $G(\veu)$ (i.e.{} $-u_a+u_b=c_{ab}$), and $aj,bj\in E$ for some $j$, then $-u_a+u_j\leq c_{aj}$ will become tight before $-u_b+u_j\leq c_{bj}$ when decreasing both $u_a$ and $u_b$ if and only if $c_{aj} \leq c_{bj}+c_{ab}$. Hence, when keeping $ab$, the latter case will never occur and only the first inequality is relevant.  On the other hand $c_{aj} > c_{bj}+c_{ab}$ implies that only $-u_b+u_j\leq c_{bj}$ can become tight, such that we only need to consider this inequality in the following. In this case we further have to adjust the value for $\vecc$ (observe $u_b= u_a+c_{ab}$).
The other case is analogous.
Hence, every edge walk (circuit walk) in $P_{G',c'}$ admits an edge walk (circuit walk) in $P_{G,c}$ that keeps the edge $ab$, such that we can continue the walk in the smaller polyhedron.

Next we show that the existence of a feasible point whose graph contains  a certain edge $ab$ implies the non-existence of a feasible point whose graph contains a different directed path from $a$ to $b$.
\begin{lemma}\label{lemmaAddIneq}
Let $P_{G,\vecc}$ be a dual network flow polyhedron.
Let $v_0v_k\in E$ such that in $G$ there is another directed path $\mathcal{P}$ from $v_0$ to $v_k$, i.e.{} there are nodes $v_0,v_1,\ldots,v_k\in V$, $k\geq 2$, such that $v_iv_{i+1}\in E$ for all $i=0,\ldots,k-1$.  

Assume there is a feasible point $\vew\in P_{G,c}$ with $v_0v_k\in E(G(\vew))$ and let $\veu\in P_{G,c}$ with $\mathcal{P}\subset G(\veu)$. Then also $v_0v_k\in E(G(\veu))$. Thus, if $P_{G,\vecc}$ is non-degenerate, there can be no such $\veu$.
\end{lemma}
\begin{proof}
The feasible point $\vew\in P_{G,c}$ satisfies
\[ c_{v_0v_k}= -w_{v_0}+w_{v_k}=\sum_{i=0}^{k-1} \left(-w_{v_i}+w_{v_{i+1}}\right) \leq \sum_{i=0}^{k-1} c_{v_iv_{i+1}}\ . \]
$\veu\in P_{G,c}$ satisfies $-u_{v_i}+u_{v_{i+1}}=c_{v_iv_{i+1}}$ for $i=0,\ldots,k-1$ and $-u_{v_0}+u_{v_k}\leq c_{v_0v_k}$. We then see
\[ \sum_{i=0}^{k-1} c_{v_iv_{i+1}} = \sum_{i=0}^{k-1} \left(-u_{v_i}+u_{v_{i+1}}\right) = -u_{v_0}+u_{v_k}\leq c_{v_0v_k} \leq \sum_{i=0}^{k-1} c_{v_iv_{i+1}}.\]
Hence, all inequalities must be satisfied with equality and we get $-u_{v_0}+u_{v_k}=c_{v_0v_k}$; that is, $v_0v_k\in E(G(\veu))$. \eoproof
\end{proof}

Observe that Lemma \ref{lemmaAddIneq} can easily be generalized to a slightly stronger statement:
Assume that there is a feasible point whose graph contains a directed path from some node $v_0$ to some node $v_k$. Then every point of the dual network flow polyhedron, whose graph contains another directed $v_0-v_k$-path, must contain the first path as well. This can only happen in the degenerate case.

\section{Proofs}\label{sec:proofs}

We begin with the proof of Theorem \ref{thm:CombDiam}. Note that for proving upper bounds on the combinatorial diameter of polyhedra it is enough to consider non-degenerate polyhedra, as by perturbation any polyhedron can be turned into a non-degenerate polyhedron whose diameter is at least as large as the one of the original polyhedron.

\mbox{}\\
{\em \noindent Proof of Theorem \ref{thm:CombDiam}.} 
Let $\veu^{(1)}$ and $\veu^{(2)}$ be two vertices of the polyhedron $P_{G,\vecc}$ given by spanning trees $T_1=G(\veu^{(1)})$ and $T_2=G(\veu^{(2)})$. We construct an edge walk from $\veu^{(1)}$ to $\veu^{(2)}$ as follows:
Being at a vertex $\vey$ of $P_{G,\vecc}$ with spanning tree $T=G(\vey)$, we choose an edge $rs\in T_2\backslash T$ we wish to insert. We show how to construct an edge walk of length at most $|E|$ that leads to a vertex $\vex$ for which $rs\in E(G(\vex))$, that is, our specified edge is added to the corresponding spanning tree. Then we contract this edge to ensure that we do not delete it again. Starting at $\vey=u^{(1)}$ and repeating this for all $|V|-1$ edges in $T_2$ proves the claimed bound of $\left(|V|-1\right)\cdot|E|$.

Now, let $\vey$ be the current vertex in our edge walk and let $T=G(\vey)$ be the corresponding spanning tree. We choose an arbitrary edge $rs\in T_2$ we wish to insert.
Given a spanning tree $T$ and the node $s$ we distinguish forward and backward edges in $E(T)$: We see $s$ as the root of the tree $T$. Then every edge in $E(T)$ lies on a unique path starting at $s$ (independent of the directions of the edges). We call the edges pointing away from $s$ \emph{backward edges}, the edges pointing towards $s$ \emph{forward edges}.

In $T$ there is a unique path (undirected) connecting $r$ and $s$.  Let $e$ be the last backward edge on this path. Note that by Lemma \ref{lemmaAddIneq} such an edge must exist. Let $R$ and $S$ be the node sets of the connected components of $T-e$ such that $r\in R$ and $s\in S$. Observe that in particular all nodes from which we can reach $s$ on a directed path in the spanning tree $T$ are assigned to $S$ (and these nodes form an arborescence of forward edges with root $s$).

We wish to include the edge $rs$ in our graph, that is, we wish to make the inequality $-u_r+u_s\leq c_{rs}$ tight. W.l.o.g. we  assume $0\in R$, therefore we add an $\epsilon$ to all components $y_i$ of $\vey$ with $i\in S$. (If $0\notin R$, we would subtract $\epsilon$ from all components $y_i$ with $i\in R$.) We choose as $\epsilon$ the smallest non-negative number such that any inequality $-u_a+u_b\leq c_{ab}$ with $a\in R$ and $b\in S$ becomes tight. Due to non-degeneracy there is only one such inequality. This creates a new feasible point $y'$, which is indeed a neighboring vertex of $\vey$ by construction. 
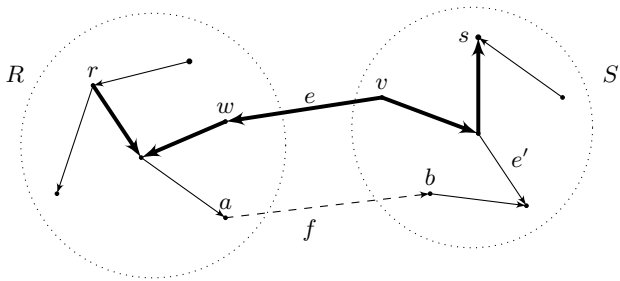
\begin{figure}[H]
		\centering
			\begin{tikzpicture}[scale=1.6]
					\coordinate (r) at (0.1,0.1);
					\coordinate (v1) at (-0.7,-0.1);
					\coordinate (v2) at (-0.3,-0.7);
					\coordinate (v3) at (0.4,-0.4);
					\coordinate (v4) at (1.7,-0.2);
					\coordinate (v5) at (2.5,-0.5);
					\coordinate (v6) at (-1,-1);
					\coordinate (v7) at (2.9,-1.1);
					\coordinate (v8) at (3.2,-0.2);
					\coordinate (a) at (0.4,-1.2);
					\coordinate (s) at (2.5,0.3);
					\coordinate (b) at (2.1,-1);				
					\draw [fill, black] (r) circle [radius=0.02];
					\draw [fill, black] (s) circle [radius=0.02];
					\draw [fill, black] (v1) circle [radius=0.015];
					\draw [fill, black] (v2) circle [radius=0.015];
					\draw [fill, black] (v3) circle [radius=0.015];
					\draw [fill, black] (v4) circle [radius=0.015];
					\draw [fill, black] (v5) circle [radius=0.015];
					\draw [fill, black] (v6) circle [radius=0.015];
					\draw [fill, black] (v7) circle [radius=0.015];
					\draw [fill, black] (v8) circle [radius=0.015];
					\draw [fill, black] (a) circle [radius=0.015];
					\draw [fill, black] (b) circle [radius=0.015];		
					\node[above] at (v1) {$r$};
					\node[above] at (v3) {$w$};
					\node[above] at (v4) {$v$};
					\node[left] at (s) {$s$};
					\node[above] at (a) {$a$};
					\node[above] at (b) {$b$};
					\node at (1.1,-0.2) {$e$};
					\node at (1.1,-1.3) {$f$};
					\node at (2.85,-0.7) {$e'$};
					\node at (-1.35,0) {$R$};
					\node at (3.6,0) {$S$};					
					\draw[->, > = latex'] (v1)--(v6);
					\draw[->,  > = latex'] (r)--(v1);
					\draw[->, > = latex', line width=1.5] (v1)--(v2);
					\draw[<-, > = latex', line width=1.5] (v2)--(v3);
					\draw[<-, > = latex', line width=1.5] (v3)--(v4);
					\draw[->,  > = latex', line width=1.5] (v4)--(v5);
					\draw[->, > = latex', line width=1.5] (v5)--(s);
					\draw[->, > = latex'] (v2)--(a);
					\draw[dashed, ->, > = latex'] (a)--(b);
					\draw[->, > = latex'] (b)--(v7);
					\draw[<-, > = latex'] (v7)--(v5);
					\draw[<-, > = latex'] (s)--(v8);					
					\draw[dotted] (-0.2,-0.6) circle (1.1);
					\draw[dotted] (2.45,-0.45) circle (1);
			\end{tikzpicture} 	
\end{figure}
So, in this edge step $e=vw$ is deleted and ${f}=ab$ is inserted.
If we inserted ${f}=rs$, we contract this edge and start over again, aiming to insert another edge $r's'$ from $E(T_2)$. 
Otherwise we consider the path connecting $r$ and $s$ in the new spanning tree $T'$. As before the last backward edge $e'$ defines sets $R'$ and $S'$ and we repeat the same procedure until eventually $rs$ is inserted. It remains to prove that this indeed happens after at most $|E|$ steps. It is enough to show that the deleted edge $e=vw$ is not inserted again: As there is a directed path from $v$ to $s$ in $G(\vey)$, $v$ and all nodes on this path will always be assigned to $S$ (in particular, no edge on this path is deleted). As only edges from $R$ to $S$ are inserted, $e=vw$ with $v\in S$ cannot be inserted twice. This proves the claimed upper bound $(|V|-1)\cdot E)$.

To see the upper bound $\frac{|V|^3}{6}$, we only have to change the way we count the number of steps that we need to insert the edge $rs$ in a current graph on $i$ nodes: Note that it has at most  $i\cdot (i-1)$ edges, and in particular at most $\binom{i}{2}$ edges $e=vw$  with $v\in S$ and $w\in R$. As we only insert edges from $R$ to $S$, this tells us an upper bound of $\binom{i}{2}$ steps until $rs$ inserted. After contracting this edge, we start this process again on a graph with $i-1$ nodes. Hence we obtain an edge walk of length at most

	\begin{align*}
			 & \sum_{i=2}^{|V|}\sum\limits_{j=1}^{i-1} j
			= \sum_{i=2}^{|V|}\left[ \frac{1}{2}i\cdot(i-1) \right] 
			= \frac{1}{2}\left[\sum_{i=2}^{|V|} i^2 - \sum_{i=2}^{|V|} i\right]
		  =& \frac{|V|^3-|V|}{6} \leq \frac{|V|^3}{6} .
	\end{align*}
 \eoproof

We continue with the proof of Theorem \ref{thm:CirDiam}. Here we cannot simply assume that the polyhedron is non-degenerate, as it is not clear whether for every degenerate polyhedron there is a perturbed non-degenerate polyhedron bounding the circuit diameter of the original one from above \cite{CircuitDiam}.
\\\\
{\em \noindent Proof of Theorem \ref{thm:CirDiam}.} 
Let $\veu^{(1)}$ and $\veu^{(2)}$ be two vertices of the polyhedron $P_{G,\vecc}$. Let $T_2$ be a spanning tree with $E(T_2)\subseteq E(G(\veu^{(2)}))$. Then $\veu^{(2)}$ is the unique point of $P_{G,\vecc}$ whose graph contains all edges in $E(T_2)$.
We construct a circuit walk from $\veu^{(1)}$ to $\veu^{(2)}$ as follows: 

Being at a point $\vey \in P_{G,\vecc}$ of our circuit walk, we choose an edge $rs\in T_2\backslash E(G(\vey))$ we wish to insert. We construct a circuit walk to a point $\vex\in P_{G,\vecc}$ with $rs\in E(G(\vex))$. This walk has length at most $i-1$, where $i$ is the number of nodes in the current underlying graph. As in the proof of Theorem \ref{thm:CombDiam}, we then contract it to make sure that we do not delete it when continuing our circuit walk.
We start with $\vey=\veu^{(1)}$ and repeat this procedure for all $|V|-1$ edges in $E(T_2)$. As the number of nodes decreases after every contraction this then yields the quadratic bound of $\sum\limits_{i=1}^{|V|-1} i= \frac{1}{2}\left(|V|\cdot (|V|-1)\right)$.

Now, let $\vey$ be a feasible point  in the circuit walk. Let $rs\in E(T_2)\backslash E(G(\vey))$ be an arbitrary edge we wish to insert, that is, we have to make $-u_r+u_s\leq c_{rs}$ tight. To this end, we construct a circuit direction that increases the component $y_s$. 
This circuit is given by $R \cupdot S=V$ for node sets $R$ and $S$ that are constructed by the following sequence of rules:
\begin{compactenum}
	\item $r$ is assigned to $R$.
	\item $s$ is assigned to $S$.
	\item All nodes from $V\backslash \{r\}$ from which $s$ can be reached on a directed path using edges in $E(G(\vey))$ are assigned to $S$. (These edges form an arborescence with root $s$.)
	\item All nodes $t\in V\backslash S$ that are connected to $r$ in the underlying undirected graph are assigned to $R$.
	\item All remaining nodes are assigned to $S$.
\end{compactenum}
Observe that from $s$ we cannot reach $r$ on a directed path in $E(G(\vey))$ by Lemma \ref{lemmaAddIneq}, hence the sets $R$ and $S$ are well-defined. Further, they satisfy all the conditions to define a circuit. Let $\veg$ be the corresponding circuit direction defined via Equation (\ref{Eq: Construction of circuit direction from R and S.}). W.l.o.g.{} we assume that $0\in R$. The case $0\in S$ works analogously by merely switching the roles of $R$ and $S$ and subtracting $\epsilon \veg$ to decrease $y_r$. 

We now apply the circuit step given by $\veg$, that is, we get the next point in our circuit walk as $\vey':=\vey+\epsilon \veg$,  where $\epsilon$ is the smallest non-negative number such that an inequality $-u_a+u_b\leq c_{ab}$ with $a\in R$ and $b\in S$ becomes tight (observe that there could be more than one such inequality, as we do not assume non-degeneracy of the polyhedron $P_{G,\vecc}$). In particular, the gap in between  $-u_r+u_s$ and its upper bound $c_{rs}$ becomes smaller. If $rs$ was indeed inserted we contract the edge and continue in a smaller polyhedron.
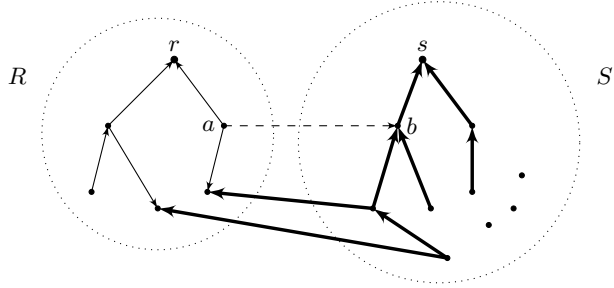
\begin{figure}[H]
		\centering
			\begin{tikzpicture}[scale=2.2]
					\coordinate (r) at (1,0.1);
					\coordinate (s) at (2.5,0.1);
					\coordinate (v1) at (2.35,-0.3);
					\coordinate (v2) at (2.8,-0.3);
					\coordinate (v3) at (2.55,-0.8);
					\coordinate (v4) at (2.2,-0.8);
					\coordinate (v5) at (2.8,-0.7);
					\coordinate (v6) at (2.65,-1.1);
					\coordinate (v7) at (3.05,-0.8);
					\coordinate (v8) at (2.9,-0.9);
					\coordinate (v9) at (3.1,-0.6);
					\coordinate (w1) at (0.6,-0.3);
					\coordinate (w2) at (1.3,-0.3);
					\coordinate (w3) at (0.9,-0.8);
					\coordinate (w4) at (0.5,-0.7);
					\coordinate (w5) at (1.2,-0.7);
					\draw [fill, black] (r) circle [radius=0.02];
					\draw [fill, black] (s) circle [radius=0.02];
					\draw [fill, black] (v1) circle [radius=0.015];
					\draw [fill, black] (v2) circle [radius=0.015];
					\draw [fill, black] (v3) circle [radius=0.015];
					\draw [fill, black] (v4) circle [radius=0.015];
					\draw [fill, black] (v5) circle [radius=0.015];
					\draw [fill, black] (v6) circle [radius=0.015];
					\draw [fill, black] (v7) circle [radius=0.015];
					\draw [fill, black] (v8) circle [radius=0.015];
					\draw [fill, black] (v9) circle [radius=0.015];
					\draw [fill, black] (w1) circle [radius=0.015];
					\draw [fill, black] (w2) circle [radius=0.015];
					\draw [fill, black] (w3) circle [radius=0.015];
					\draw [fill, black] (w4) circle [radius=0.015];
					\draw [fill, black] (w5) circle [radius=0.015];					
					\node[above] at (r) {$r$};
					\node[above] at (s) {$s$};
					\node[right] at (v1) {$b$};
					\node[left] at (w2) {$a$};					
					\node at (0.05,0) {$R$};
					\node at (3.6,0) {$S$};					
					\draw[<-, > = latex', line width= 1.3] (s)--(v1);
					\draw[<-,  > = latex', line width= 1.3] (s)--(v2);
					\draw[<-, > = latex', line width= 1.3] (v1)--(v3);
					\draw[<-, > = latex', line width= 1.3] (v1)--(v4);
					\draw[<-,  > = latex', line width= 1.3] (v2)--(v5);
					\draw[<-,  > = latex', line width= 1.3] (v4)--(v6);
					\draw[<-,  > = latex'] (r)--(w1);
					\draw[->,  > = latex'] (w2)--(r);
					\draw[<-,  > = latex'] (w3)--(w1);
					\draw[<-,  > = latex'] (w1)--(w4);
					\draw[<-,  > = latex'] (w5)--(w2);
					\draw[->,  > = latex', dashed] (w2)--(v1);
					\draw[->,  > = latex', line width= 1.3] (v4)--(w5);
					\draw[->,  > = latex', line width= 1.3] (v6)--(w3);
					
					\draw[dotted] (0.9,-0.35) circle (0.7);
					\draw[dotted] (2.6,-0.4) circle (0.85);
			\end{tikzpicture} 	
\end{figure}
Otherwise, the inserted edge extends the arborescence by at least the node $a$.
We again apply a circuit step by constructing sets $R'$ and $S'$ for $y'$ as before, which inserts $rs$ or extends the arborescence further. 
Continuing like this after at most $i-2$ steps all nodes but $r$ are contained in the arborescence (if $rs$ was not already inserted). Then the next step must add $rs$ by Lemma \ref{lemmaAddIneq}. \eoproof

Observe that these diameter bounds also hold for dual network flow polyhedra defined on directed graphs that are not connected. To make the polyhedron pointed, we set, for each connected component, the value of one variable to zero (just as we fixed $u_0=0$ for connected graphs with just one connected component). Then the algorithmic approaches described in the proofs of Theorem \ref{thm:CombDiam} and Theorem \ref{thm:CirDiam} can be applied to each connected component individually, yielding even better bounds on the combinatorial diameter and the circuit diameter.

\section{Lower bounds}\label{sec:ex}

In the above, we derived quadratic upper bounds on the circuit and the combinatorial diameter of dual transportation polyhedra. We now complement our discussion by constructing an infinite family of graphs that exhibit that the gap between the number of nodes $|V|$ and the circuit diameter of a polyhedron associated with a certain graph can be arbitrarily large.

To this end, we begin with a formal introduction of a \emph{glueing construction} for graphs:
Let $G_i=\left(V_i,E_i\right)$, $i=1,\ldots,k$ be $k$ connected directed graphs. For every graph choose an arbitrary node $v_0^i\in V_i$. We construct a new graph $G=\left(V,E\right)$ by glueing the graphs together at the $v_0^i$, joining them to one node $v_0$. Formally, the node sets and the edge set are given by
\begin{align*}
V:=& \left\{v_0\right\}\cup \bigcup_{i=1}^k \left(V_i\backslash \left\{v_0^i\right\}\right) \\
E:=& \bigcup_{i=1}^k \left(\, \left\{ab:\ ab\in E_i, a,b\neq v^i_0\right\} \cup \left\{v_0b: v_0^ib\in E_i \right\}\cup \left\{av_0: av_0^i\in E_i \right\}\, \right).
\end{align*}
We depict the graphs $G_i$ by highlighting the nodes $v_0^i$, while all remaining nodes and edges are represented by a cycle:

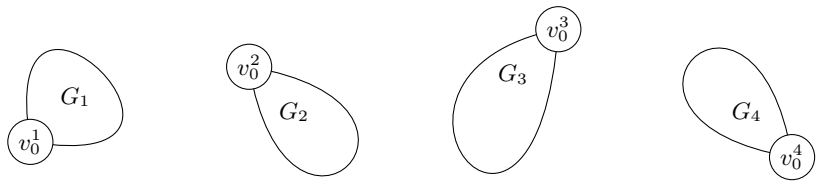
\begin{figure}[H]
\centering
\begin{tikzpicture}
\useasboundingbox (-1,0) rectangle (2,1) ;
	 \node at (0.6,0.6) {$G_1$};
   \draw (0,0) .. controls (-0.5,3) and  (3,-0.5) ..  (0,0); 
	 \node (v0) at  (0,0) {$v_0^1$};
	 \draw[fill= white] (v0) circle  (0.3);
	 \node at (v0) {$v_0^1$};
 \end{tikzpicture}
\begin{tikzpicture}
\useasboundingbox (-0.8,-1) rectangle (0,1) ;
	 \node at (0.6,-0.6) {$G_2$};
	 \draw (0,0) .. controls (0.5,-3) and  (3,-0.5) ..  (0,0); 
	 \node (v0) at  (0,0) {$v_0^2$};
	 \draw[fill= white] (v0) circle  (0.3);
	 \node at (v0) {$v_0^2$};
 \end{tikzpicture}
\begin{tikzpicture}
\useasboundingbox (-4,-1.5) rectangle (0,0) ;
	 \node at (-0.6,-0.6) {$G_3$};
	 \draw (0,0) .. controls (-0.3,-4) and  (-3,-0.6) ..  (0,0); 
	 \node (v0) at  (0,0) {$v_0^3$};
	 \draw[fill= white] (v0) circle  (0.3);
	 \node at (v0) {$v_0^3$};
 \end{tikzpicture}
\begin{tikzpicture}
\useasboundingbox (-3,0.2) rectangle (-0.5,3) ;
	 \node at (-0.6,0.6) {$G_4$};
   \draw (0,0) .. controls (-0.5,3) and  (-3,0.5) ..  (0,0); 
	 \node (v0) at  (0,0) {$v_0^4$};
 	 \draw[fill= white] (v0) circle  (0.3);
	 \node at (v0) {$v_0^4$};
 \end{tikzpicture}
\end{figure}
Glueing these $4$ graphs together yields a graph $G$ that can be illustrated as follows:

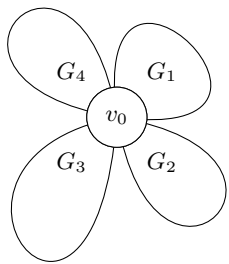
\begin{figure}[H]
\centering
\begin{tikzpicture}
\useasboundingbox (-2,-2) rectangle (2,2) ;
  \node (v0) at  (0,0) {$v_0$};
	\draw (v0) circle (0.4);
 	\node at (0.6,0.6) {$G_1$};
	\draw (0,0) .. controls (-0.4,3) and  (3,-0.4) ..  (0,0); 
	\node at (0.6,-0.6) {$G_2$};
	\draw (0,0) .. controls (0.5,-3) and  (3,-0.5) ..  (0,0); 
	\node at (-0.6,-0.6) {$G_3$};
	\draw (0,0) .. controls (-0.3,-4) and  (-3,-0.6) ..  (0,0); 
	\node at (-0.6,0.6) {$G_4$};
	\draw (0,0) .. controls (-0.5,3) and  (-3,0.5) ..  (0,0); 
	\draw[fill= white] (v0) circle  (0.4);
  \node (v0) at  (0,0) {$v_0$};
 \end{tikzpicture}
\end{figure}

Now the diameters of the polyhedra associated to these graphs are directly related.
\begin{lemma}\label{lemmaGlueing}
Let $P_{G_i,\vecc^i}$, $i=1,\ldots,k$ be arbitrary dual network flow polyhedra with combinatorial (circuit) diameter equal to $d_i$, respectively at least $d_i$.
Let $G$ be the graph obtained by glueing these $k$ graphs together, and define $c\in \R^{|E|}$ by $c_{lj}= c^i_{lj}, lj\in E_i$.

Then $P_{G,\vecc}$ has combinatorial (circuit) diameter $\sum_{i=1}^k d_i$, respectively at least $\sum_{i=1}^k d_i$.
\end{lemma}
\begin{proof}
Let a circuit direction of $P_{G,\vecc}$ be given by a partition $V=R\cupdot S$. Assume w.l.o.g.{} $v_0\in R$. Then $S\subseteq V_i\backslash\left\{v_0^i\right\}$ for some $i\in \left\{1,\ldots,k\right\}$, as the node set $S$ must be connected in the underlying graph and $v_0\notin S$.

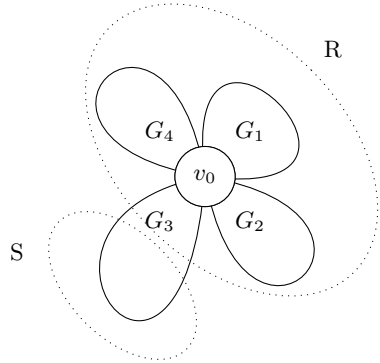
\begin{figure}[H]
\centering
\begin{tikzpicture}
\useasboundingbox (-2,-2) rectangle (2,2) ;
  \node (v0) at  (0,0) {$v_0$};
	\draw (v0) circle (0.4);
 	\node at (0.6,0.6) {$G_1$};
	\draw (0,0) .. controls (-0.4,3) and  (3,-0.4) ..  (0,0); 
	\node at (0.6,-0.6) {$G_2$};
	\draw (0,0) .. controls (0.5,-3) and  (3,-0.5) ..  (0,0); 
	\node at (-0.6,-0.6) {$G_3$};
	\draw (0,0) .. controls (-0.3,-4) and  (-3,-0.6) ..  (0,0); 
	\node at (-0.6,0.6) {$G_4$};
	\draw (0,0) .. controls (-0.5,3) and  (-3,0.5) ..  (0,0); 
	\draw[fill= white] (v0) circle  (0.4);
  \node (v0) at  (0,0) {$v_0$};
  \draw[rotate= 45, dotted] (0.5,0) ellipse  (1.5 and 2.3);
  \draw[rotate= 45, dotted] (-1.8,-0.25) ellipse  (0.7 and 1.2);
  \node at (1.7,1.7) {R};
  \node at (-2.5,-1) {S};
 \end{tikzpicture}
\end{figure}

Therefore, every step of an edge walk (circuit walk) modifies only variables corresponding to a single, particular component $G_i$, such that every edge walk (circuit walk) on $P_{G,\vecc}$ of length $d'$ directly translates into $k$ edge walks (circuit walks) on $P_{G_1,\vecc^1},\ldots,P_{G_k,\vecc^k}$ of length $d'_1,\ldots,d'_k$  with $\sum_{i=1}^k d'_i=d'$ and vice versa. \eoproof
\end{proof}

We now turn to an example which shows that there are configurations in which there is no circuit step that inserts an edge from the target tree. Note that in the undirected bipartite case we are always able to apply such a step.
Therefore, recall Example \ref{THEex} in which we introduced the polyhedron $P_{G,\vecc}$ defined on the following graph. The labels on the edges correspond to the values of $\vecc$.
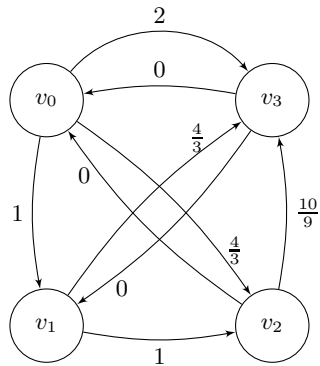
\begin{figure}[H]
\centering
\begin{tikzpicture}
	\tikzset{vertex/.style = {shape=circle,draw,minimum size=3em}}
  \node[vertex] (v0) at  (0,3) {$v_0$};
  \node[vertex] (v1) at  (0,0) {$v_1$};
  \node[vertex] (v2) at  (3,0) {$v_2$};
  \node[vertex] (v3) at  (3,3) {$v_3$};
  
  \tikzset{edge/.style = {->,> = latex'}}
  \draw[edge, ->] (v3) [bend right = 10] to  node[above] {$0$} (v0);
  \draw[edge, ->] (v2) [bend left = 10] to (v0);
  \node at (0.5,2) {$0$};
  \draw[edge, ->] (v3) [bend left = 10] to (v1);
	\node at (1,0.5) {$0$};
  \draw[edge, ->] (v0) [bend left = 50] to [above] node {$2$} (v3);
  \draw[edge, ->] (v0) [bend left=10] to (v2);
  \node at (2.5,1) {$\frac{4}{3}$};
  \draw[edge, ->] (v1) [bend left=10] to (v3);
  \node at (2,2.5) {$\frac{4}{3}$};
  \draw[edge, ->] (v0) [bend right=10] to [left] node {$1$} (v1);
  \draw[edge, ->] (v1) [bend right=10] to [below] node {$1$} (v2);
  \draw[edge, ->] (v2) [bend right=10] to [right] node {$\frac{10}{9}$} (v3);
\end{tikzpicture}
\end{figure}
Observe that the polyhedron $P_{G,\vecc}$ is non-degenerate (there can be no cycle of tight inequalities).
The following two spanning trees correspond to vertices $\veu^{(1)}$ and $\veu^{(2)}$  of $P_{G,\vecc}$. The nodes are labeled by the values of the corresponding variables.
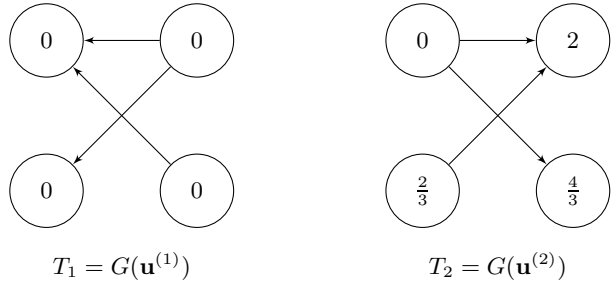
\begin{figure}[H]
\centering
\begin{tikzpicture}
	\tikzset{vertex/.style = {shape=circle,draw,minimum size=3em}}
  \node[vertex] (v0) at  (0,2) {$0$};
  \node[vertex] (v1) at  (0,0) {$0$};
  \node[vertex] (v2) at  (2,0) {$0$};
  \node[vertex] (v3) at  (2,2) {$0$};
  \node at (1,-1) {$T_1=G(\veu^{(1)})$};
  
  \node[vertex] (w0) at  (5,2) {$0$};
  \node[vertex] (w1) at  (5,0) {$\frac{2}{3}$};
  \node[vertex] (w2) at  (7,0) {$\frac{4}{3}$};
  \node[vertex] (w3) at  (7,2) {$2$};
  \node at (6,-1) {$T_2=G(\veu^{(2)})$};
  
  \tikzset{edge/.style = {->,> = latex'}}
  \draw[edge, ->] (v3)  to   (v0);
  \draw[edge, ->] (v2) to (v0);
  \draw[edge, ->] (v3) to (v1);
  \draw[edge, ->] (w0) to (w3);
  \draw[edge, ->] (w0) to (w2);
  \draw[edge, ->] (w1) to (w3);
  \end{tikzpicture}
\end{figure}
These two vertices are connected via the following edge walk of length $4$. Hence their circuit distance and combinatorial distance are at most $4$.
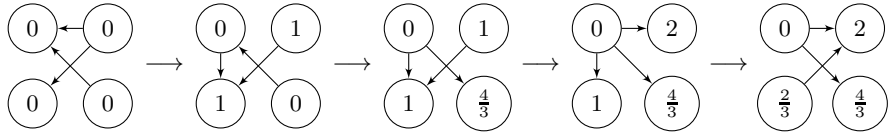
\begin{figure}[H]
\centering
\begin{tikzpicture}
	\tikzset{vertex/.style = {shape=circle,draw,minimum size=2em}}
  \tikzset{edge/.style = {->,> = latex'}}

  \node[vertex] (v00) at  (0,1) {$0$};
  \node[vertex] (v10) at  (0,0) {$0$};
  \node[vertex] (v20) at  (1,0) {$0$};
  \node[vertex] (v30) at  (1,1) {$0$};
	\draw[edge, ->] (v20) to (v00);
  \draw[edge, ->] (v30) to (v00);
  \draw[edge, ->] (v30) to (v10);
  
  \node[vertex] (v01) at  (2.5,1) {$0$};
  \node[vertex] (v11) at  (2.5,0) {$1$};
  \node[vertex] (v21) at  (3.5,0) {$0$};
  \node[vertex] (v31) at  (3.5,1) {$1$};
	\draw[edge, ->] (v01) to (v11);
  \draw[edge, ->] (v21) to (v01);
  \draw[edge, ->] (v31) to (v11);
  
  \node[vertex] (v02) at  (5,1) {$0$};
  \node[vertex] (v12) at  (5,0) {$1$};
  \node[vertex] (v22) at  (6,0) {\small$\frac{4}{3}$};
  \node[vertex] (v32) at  (6,1) {$1$};
	\draw[edge, ->] (v02) to (v12);
  \draw[edge, ->] (v02) to (v22);
  \draw[edge, ->] (v32) to (v12);

  \node[vertex] (v03) at  (7.5,1) {$0$};
  \node[vertex] (v13) at  (7.5,0) {$1$};
  \node[vertex] (v23) at  (8.5,0) {\small$\frac{4}{3}$};
  \node[vertex] (v33) at  (8.5,1) {$2$};
	\draw[edge, ->] (v03) to (v13);
  \draw[edge, ->] (v03) to (v23);
  \draw[edge, ->] (v03) to (v33);

  \node[vertex] (v04) at  (10,1) {$0$};
  \node[vertex] (v14) at  (10,0) {\small$\frac{2}{3}$};
  \node[vertex] (v24) at  (11,0) {\small$\frac{4}{3}$};
  \node[vertex] (v34) at  (11,1) {$2$};
	\draw[edge, ->] (v04) to (v34);
  \draw[edge, ->] (v04) to (v24);
  \draw[edge, ->] (v14) to (v34);
  
 	\node at (1.75,0.5) {$\longrightarrow$};
	\node at (4.25,0.5) {$\longrightarrow$};
	\node at (6.75,0.5) {$\longrightarrow$};
	\node at (9.25,0.5) {$\longrightarrow$};
\end{tikzpicture}
\end{figure}
We now illustrate all possible first circuit steps from $\veu^{(1)}$, leading to points $\vey^{(1)},\ldots,\vey^{(6)}$. The corresponding circuits are stated below the graphs and are w.l.o.g.{} given by subsets $S\subseteq V$ such that $v_0\notin S$ (note that $S=\{v_1,v_2\}$ is not applicable). Observe that in all cases the inserted (bold) edge is not in $E(T_2)$.
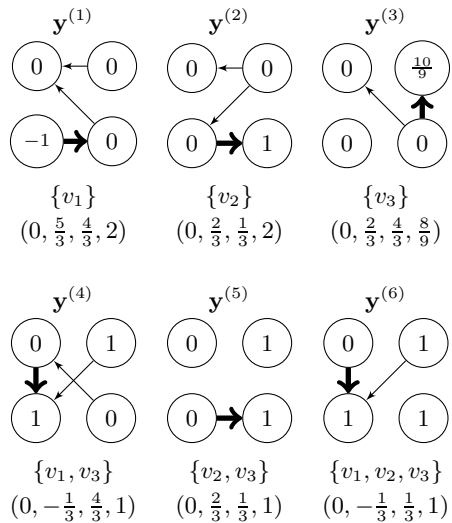
\begin{figure}[H]
\centering
\begin{tabular} {ccccc} 
\\
$\vey^{(1)}$ && $\vey^{(2)}$ && $\vey^{(3)}$\\ 
\begin{tikzpicture}
	\tikzset{vertex/.style = {shape=circle,draw,minimum size=2em}}
  \node[vertex] (v0) at  (0,1) {$0$};
  \node[vertex] (v1) at  (0,0) {\scriptsize$-1$};
  \node[vertex] (v2) at  (1,0) {$0$};
  \node[vertex] (v3) at  (1,1) {$0$};
  \tikzset{edge/.style = {->,> = latex'}}
  \draw[edge, ->] (v3) to (v0);
  \draw[edge, ->] (v2) to (v0);
  \draw[ ->, line width= 2] (v1) to (v2);
\end{tikzpicture}
& \hspace{0.2cm} &
\begin{tikzpicture}
	\tikzset{vertex/.style = {shape=circle,draw,minimum size=2em}}
  \node[vertex] (v0) at  (0,1) {$0$};
  \node[vertex] (v1) at  (0,0) {$0$};
  \node[vertex] (v2) at  (1,0) {$1$};
  \node[vertex] (v3) at  (1,1) {$0$};
  \tikzset{edge/.style = {->,> = latex'}}
  \draw[edge, ->] (v3) to (v0);
  \draw[->, line width=2] (v1) to (v2);
  \draw[edge, ->] (v3) to (v1);
\end{tikzpicture}
& \hspace{0.2cm} &
\begin{tikzpicture}
	\tikzset{vertex/.style = {shape=circle,draw,minimum size=2em}}
  \node[vertex] (v0) at  (0,1) {$0$};
  \node[vertex] (v1) at  (0,0) {$0$};
  \node[vertex] (v2) at  (1,0) {$0$};
  \node[vertex] (v3) at  (1,1) {\tiny$\frac{10}{9}$};
  \tikzset{edge/.style = {->,> = latex'}}
  \draw[ ->, line width= 2] (v2) to (v3);
  \draw[edge, ->] (v2) to (v0);
\end{tikzpicture}\\
$\{v_1\}$ && $\{v_2\}$ && $\{v_3\}$\\
$(0,\frac{5}{3},\frac{4}{3},2)$ &&
$(0,\frac{2}{3},\frac{1}{3},2)$ &&
$(0,\frac{2}{3},\frac{4}{3},\frac{8}{9})$\\
\\
$\vey^{(4)}$ && $\vey^{(5)}$ && $\vey^{(6)}$\\ 
\begin{tikzpicture}
	\tikzset{vertex/.style = {shape=circle,draw,minimum size=2em}}
  \node[vertex] (v0) at  (0,1) {$0$};
  \node[vertex] (v1) at  (0,0) {$1$};
  \node[vertex] (v2) at  (1,0) {$0$};
  \node[vertex] (v3) at  (1,1) {$1$};
  \tikzset{edge/.style = {->,> = latex'}}
	\draw[ ->, line width= 2] (v0) to (v1);
  \draw[edge, ->] (v2) to (v0);
  \draw[edge, ->] (v3) to (v1);
\end{tikzpicture}
& \hspace{0.2cm} &
\begin{tikzpicture}
	\tikzset{vertex/.style = {shape=circle,draw,minimum size=2em}}
  \node[vertex] (v0) at  (0,1) {$0$};
  \node[vertex] (v1) at  (0,0) {$0$};
  \node[vertex] (v2) at  (1,0) {$1$};
  \node[vertex] (v3) at  (1,1) {$1$};
  \tikzset{edge/.style = {->,> = latex'}}
  \draw[ ->, line width= 2] (v1) to (v2);
\end{tikzpicture}
& \hspace{0.2cm} &
\begin{tikzpicture}
	\tikzset{vertex/.style = {shape=circle,draw,minimum size=2em}}
  \node[vertex] (v0) at  (0,1) {$0$};
  \node[vertex] (v1) at  (0,0) {$1$};
  \node[vertex] (v2) at  (1,0) {$1$};
  \node[vertex] (v3) at  (1,1) {$1$};
  \tikzset{edge/.style = {->,> = latex'}}
	\draw[->, line width= 2] (v0) to (v1);
  \draw[edge, ->] (v3) to (v1);  
\end{tikzpicture}
\\
$\{v_1,v_3\}$ && $\{v_2,v_3\}$ && $\{v_1,v_2,v_3\}$ 
\\
$(0,-\frac{1}{3},\frac{4}{3},1)$ &&
$(0,\frac{2}{3},\frac{1}{3},1)$ &&
$(0,-\frac{1}{3},\frac{1}{3},1)$ 
\\
\end{tabular}
\end{figure}
\medskip

It then is elementary to verify that the circuit distance from $\veu^{(1)}$ to $\veu^{(2)}$ is indeed $|V|=4$. For this purpose, it is sufficient to see that one was not able to insert an edge from $E(T_2)$ in the first circuit step, and that in the remaining circuit walk we cannot insert two edges from $E(T_2)$ at the same time. Even if the latter property would not hold for a given $\vecc$, we could always satisfy it by a slight perturbation:

For every single step of a circuit walk, a finite number of linear conditions on the right-hand sides $\vecc$ guarantees that only at most one edge from a target tree is inserted. Thus, after $k$ steps on a circuit walk, we only have to exclude the $\vecc$ in the union of a countable number of hyperplanes to be able to guarantee this property for all steps of circuit walks of length at most $k$.

So we now have a graph $G$ with circuit distance (at least) $|V|=4$. Applying Lemma \ref{lemmaGlueing} to $k$ copies $G_i$ of $G$ yields a new graph $G^k$ on $3k+1$ nodes with (combinatorial and circuit) diameter at least $4k$. This gives us a family of graphs $G^k$ for which both diameters exceed the number of nodes by an arbitrary constant $k-1$ and the ratio between the diameters and the number of nodes approaches $\frac{4}{3}$ for $k\to \infty$. In particular we get the following lower bound statement for the circuit diameter (and hence also for the combinatorial diameter) of dual network flow polyhedra.
\begin{lemma}
For any $n\geq 4$, there is a graph $G=(V,E)$ on $|V|=n$ nodes and a vector $\vecc \in \R^{|E|}$ such that 
	\[ \diam_\Circuits\left(P_{G,\vecc}\right) \geq \frac{4}{3} |V|-4 \ .\]
\end{lemma}
\begin{proof}
For $n=3k+1$ with $k\in \Z$ the claim follows by choosing $G=G^k$, as $k=\frac{|V|-1}{3}$ and the circuit diameter is at least $4k$. If $n=3k+2$  ($n=3k+3$) we simply add one leaf (two leaves) to $G^k$. Then $k=\frac{|V|-2}{3}$ ($k=\frac{|V|-3}{3}$) and the circuit diameter is again at least $4k$.
\eoproof
\end{proof}

\bigskip

\section*{Acknowledgments} 

The second author gratefully acknowledges the support from the graduate program TopMath of the Elite Network of Bavaria and the TopMath Graduate Center of TUM Graduate School at Technische Universit\"at M\"unchen.

\bibliographystyle{plain}
\bibliography{biblioAugmentation}

\end{document}